\documentclass[12pt,reqno]{article}

\usepackage[usenames]{color}
\usepackage{amssymb}
\usepackage{graphicx}
\usepackage{amscd}

\usepackage[colorlinks=true,
linkcolor=webgreen,
filecolor=webbrown,
citecolor=webgreen]{hyperref}

\definecolor{webgreen}{rgb}{0,.5,0}
\definecolor{webbrown}{rgb}{.6,0,0}

\usepackage{color}
\usepackage{fullpage}
\usepackage{float}

\usepackage{graphics,amsmath,amssymb}
\usepackage{amsthm}
\usepackage{amsfonts}
\usepackage{latexsym}
\usepackage{epsf}

\usepackage{verbatim}
\usepackage{enumerate}

\usepackage{tikz}
\usepackage{caption}
\usepackage{capt-of}
\usepackage{forest}

\usepackage{mathtools}

\usepackage{array}

\usepackage[blocks]{authblk}

\makeatletter
\newcommand{\thickhline}{
    \noalign {\ifnum 0=`}\fi \hrule height 1pt
    \futurelet \reserved@a \@xhline
}
\newcolumntype{"}{@{\hskip\tabcolsep\vrule width 1pt\hskip\tabcolsep}}
\makeatother

\newcommand{\seqnum}[1]{\href{http://oeis.org/#1}{\underline{#1}}}

\theoremstyle{plain}
\newtheorem{theorem}{Theorem}

\newtheorem{lemma}[theorem]{Lemma}
\newtheorem{proposition}[theorem]{Proposition}

\theoremstyle{definition}

\theoremstyle{remark}

\begin{document}

\title{Sequences of the Stable Matching Problem}
\author{Matvey Borodin}
\author{Eric Chen}
\author{Aidan Duncan}
\author{Boyan Litchev}
\author{Jiahe Liu}
\author{Veronika Moroz}
\author{Matthew Qian}
\author{Rohith Raghavan}
\author{Garima Rastogi}
\author{Michael Voigt}
\affil{PRIMES STEP}
\author{Tanya Khovanova}
\affil{MIT}

\maketitle

\begin{abstract}
In this paper, we begin by discussing different types of preference profiles related to the stable marriage problem. We then introduce the concept of soulmates, which are a man and a woman who rank each other first. Inversely, we examine hell-pairs, where a man and a woman rank each other last. We generate sequences enumerating preference profiles of different types. We also calculate sequences related to the egalitarian cost, or ``quality'', of a matching. In total, we introduce and discuss 30 new sequences related to the stable marriage problem and discuss 6 sequences that are already in the OEIS.
\end{abstract}

\section{Introduction}

The stable marriage problem (SMP) \cite{GS1962} is a problem that marries $n$ men and $n$ women depending on their preferences in such a way that the matching is \textit{stable}. A matching is not stable if there are two people of opposite genders who prefer each other over their assigned partners.

The problem applies to a variety of settings and has a lot of applications and generalizations. The most famous application is for matching graduate medical students to their first hospitals \cite{GS1962}, which is a variation of the stable marriage problem, where hospitals are allowed to hire several graduate students.

In this paper, we describe a lot of sequences related to this problem.

In Section~\ref{sec:prel} we set our assumptions restricting ourselves to the original setting of the stable marriage problem. We describe preliminary definitions, including the notion of the egalitarian cost for matching a pair of people of opposite genders.

In Section~\ref{sec:pp} we introduce sequences related to different types of preference profiles. These sequences include the total number of preference profiles, as well as the profiles where all men prefer the same woman, the first man is more desirable than the second man, all men rank the women the same way, all women have distinct men as their first choice, and the matrix of men's profiles form a Latin square.

In Section~\ref{sec:soulmates} we concentrate on sequences that depend on the number of soulmates. Soulmate pairs are an important notion in the SMP, as they describe a pair of people that rank each other first. As a result, any soulmate pairs in a stable matching have to be married; after that, they can be removed from consideration. We count the number of profiles with $k$ pairs of soulmates. We also count the number of profiles with $k$ pairs of soulmates in a special case when the men's and women's preferences form Latin squares. We provide formulas for these sequences.

In Section~\ref{sec:distSM}, we calculate the number of profiles that generate a given number of stable matchings. There is no easy formula for these numbers, so they are calculated by a program. Given the computational complexity, we can only calculate the values for $n$ up to 4. For $n=2$, the number of stable matchings could be 1 or 2. For $n=3$, the number of stable matchings could be 1, 2, or 3. For $n=4$, the number of stable matchings is not more than 10.

In Section~\ref{sec:special} we consider some interesting sequences. First, we look at hell-pairs, which are pairs of people where the men and women rank each other last. Intriguingly, they may end up together in a stable matching, in which case we call them a hell-couple. We count the number of profiles where this is possible. Then we count the number of profiles containing outcasts---pairs of people whom everyone else ranks the worst, making them destined for each other in a stable marriage. Finally, we count the number of profiles where the outcasts form a hell-couple.

Section~\ref{sec:symmetries} is devoted to symmetries of the problem. We count the number of preference profiles up to relabeling men.

In Section~\ref{sec:ec}, for a given $n$, we count the number of profiles such that there exists a stable matching of given egalitarian costs. In addition, we count the total number of stable matchings for a given egalitarian cost.

\section{Preliminaries}\label{sec:prel}

First, we make several assumptions:

\begin{itemize}
\item We have $n$ men and $n$ women.
\item Each person prefers to be married over being single.
\item A man only marries a woman, and vice versa.
\item Each person ranks people of the opposite gender without ties.
\end{itemize}

Thus, each man numbers the $n$ women in increasing order from $1$ through $n$, with 1 being his favorite and $n$ being his least favorite. Similarly, each woman numbers the $n$ men in increasing order from $1$ through $n$, with 1 being her favorite and $n$ being her least favorite. We call these $2n$ sets of preferences a \textit{preference profile}.

We can look separately at men's profiles and women's profiles. We can arrange men's preferences in a matrix where the $i$-th row corresponds to the $i$-th man's preferences. Thus, in such a matrix, each row is a permutation of the numbers 1 through $n$. Similar arguments work for women's preferences.

Once all preference profiles are set, people are paired into marriages. We call the set of all these marriages \textit{unstable} if there is a man $M$ and woman $W$ such that they prefer each other to their spouses, and \textit{stable} if no such man/woman pair exists. In an unstable set of marriages, the pair $M$ and $W$ described above is called a \textit{blocking pair} or a \textit{rogue couple}. By definition, a set of marriages without a rogue couple is stable.

We call a set of stable marriages a \textit{stable matching}. We denote the maximum number of stable matchings for $n$ men and $n$ women as $M(n)$. If two people can end up married in a stable matching, they are called \textit{valid partners} of each other.

The Gale-Shapley algorithm \cite{GS1962}, also known as the deferred acceptance algorithm, finds a stable matching given a preference profile. It involves a number of rounds described as follows:
\begin{itemize}
\item In the first round, each unengaged man proposes to the woman he prefers most, and then each woman replies ``maybe'' to her most preferable suitor and ``no'' to all other suitors. Each woman is then provisionally ``engaged'' to the suitor she prefers the most of those who proposed to her in that round. This suitor is, likewise, provisionally engaged to her.
\item In each subsequent round, each unengaged man proposes to the most-preferred woman to whom he has not yet proposed, regardless of whether the woman is already engaged. Then, each woman becomes provisionally engaged to the suitor she prefers the most out of those who have proposed to her so far. In doing so, she may reject her current provisional partner, who then becomes unengaged.
\item This process is repeated until everyone is engaged.
\end{itemize}

This version of the algorithm is called \textit{men-proposing}. By symmetry, there exists a women-proposing algorithm. The men-proposing Gale-Shapley algorithm is \textit{man-optimal} and \textit{woman-pessimal}. That means every man gets the best possible valid partner, and every woman gets the worst possible one. By symmetry, the women-proposing Gale-Shapley algorithm is \textit{woman-optimal} and \textit{man-pessimal}.

\subsection{The egalitarian cost}

The \textit{pairwise egalitarian cost} of a man and a woman is the sum of the rankings they give each other \cite{GI1989}. Suppose woman $W$ ranks man $M$ as $i$-th in her preferences, and $M$ ranks $W$ as $j$-th in his preferences. Then, the egalitarian cost of the pair $(M,W)$ is equal to $i + j$. The \textit{egalitarian cost} of a stable matching is the sum of all the pairwise egalitarian costs of the married couples in the matching. A matching with the smallest possible egalitarian cost is called \textit{egalitarian}.

Now, suppose we have two people with an egalitarian cost of 2. This means that they both rank each other first. Thus, unless they are married, they will form a blocking pair. That means any stable marriage has to have them as a couple. We call such a pair a \textit{soulmate pair} or \textit{soulmates}.

Soulmates are important as they reduce the matching problem to fewer people. In a stable matching problem with $n$ men, $n$ women, and $k$ pairs of soulmates, the number of stable matchings is the same as the number for the group of people from which the soulmates couples are removed, with the preferences for each person adjusted appropriately. In particular, the maximum number of stable marriages for any set of $2n$ people with $k$ pairs of soulmates cannot exceed $M(n - k)$.

\section{Preference profiles}\label{sec:pp}

\subsection{The total number of profiles}

The total number of different preference profiles is 
\[(n!)^{2n}.\]
This is because each person can rank the opposite gender in $n!$ ways. 

Here are the first few terms of the sequence for the total number of different preference profiles, with the first term corresponding to $n = 1$:
\[1,\ 16,\ 46656,\ 110075314176,\ 619173642240000000000.\]
This is sequence \seqnum{A185141} in the OEIS \cite{OEIS}, where it is defined as ``$a(n)$ is the number of `templates', or ways of placing a single digit within an $n^2$ by $n^2$ Sudoku puzzle so that all rows, columns, and $n$ by $n$ blocks have exactly one copy of the digit.'' This connection between preference profiles and Sudokus motivated the other paper we wrote \cite{STEPSudoku}.

\subsection{The homecoming queen}

We consider profiles where all men prefer the same woman. We call this woman a \textit{homecoming queen}. We count the profiles where a homecoming queen exists. Equivalently, we can consider the profiles where all women have the same man, a homecoming king, as their first choice.

For $n$ men and $n$ women, there are $n$ ways to pick the woman that all men rank first. Then there are $(n-1)!^n$ ways to fill in the rest of their preference profiles. There are then $n!^n$ ways to fill in the preference profiles for the women. This results in a total of 
\[n(n-1)!^nn!^n\]
profiles where all men prefer the same woman. This is the same number of preference profiles as when there exists a woman that all men rank at the $i$-th place, where $i$ can be anywhere from 1 to $n$. 

The sequence of the total number of preference profiles where all men prefer the same woman as their first choice is now sequence \seqnum{A340890} and starts as
\[1,\ 8,\ 5184,\ 1719926784,\ 990677827584000000,\ 2495937495082991616000000000000,\ \ldots. \]

We can ignore women's preferences and divide the sequence by $n!^n$. We get the following sequence which counts the number of different possible men's preferences in the stable marriage problem where all men prefer the same women. This is now sequence \seqnum{A342573}:
\[1,\ 2,\ 24,\ 5184,\ 39813120,\ 17915904000000,\ \ldots. \]
The formula for the $n$-th term is 
\[n(n-1)!^n.\]
This sequence can be viewed as the number of ordered $n$-tuples consisting of $n$ permutations of order $n$ (not necessarily distinct) such that the first element of each of them is the same. Equivalently, we can say that the $i$-th element of each of them is the same.

We can also fix the woman everyone prefers. That is, we can assume that the homecoming queen is woman number 1. Then we get the sequence $a(n) = (n-1)!^n$. 
\[1, 1, 8, 1296, 7962624, 2985984000000, 100306130042880000000,\ \ldots. \] 
This sequence is the same as sequence \seqnum{A091868} shifted by 1.

Another natural sequence is the number of profiles when all men prefer the same woman and all women prefer the same man. To calculate it, we just need to square the corresponding terms in sequence A342573, which is the number of different possible men's preferences where they all prefer the same woman as their first choice. We get 
\[1, 4, 576, 26873856, 1585084524134400, 320979616137216000000000000,\ \ldots. \]

This is now sequence \seqnum{A343474}: the number of preference profiles where both all men and all women have the same person as their first choice. Every preference profile of this type has a homecoming queen and a homecoming king. They rank each other first; that is, they are soulmates. Moreover, such profiles have exactly one pair of soulmates.

\subsection{One woman is preferable than another}

The homecoming queen is a more desirable woman than any other woman. We want to relax this condition, and consider a pair of women, where one woman is always ranked higher than the other one.

We see that for the men's profiles, there are $\frac{n!}{2}$ ways to make each profile, as only half of them will meet our conditions. Note that we are assuming that $n>1$. This means that the total number of ways to arrange the men's profiles is $\frac{n!^n}{2^n}$. Since there are no conditions for the women's profiles, there are $n!^n$ ways to arrange them. This gives us a total of 
\[\frac{n!^{2n}}{2^n}\]
profiles where there exist two women, A and B, such that every man prefers A over B. This is now sequence \seqnum{A338665}, and it starts as
\[4,\ 5832,\ 6879707136,\ 19349176320000000000,\ \ldots.\]

Each man prefers woman A to woman B, so they will always propose to woman A first. When woman B is proposed to, the man proposing to woman B has already proposed to woman A. Since a man can only propose once per round, woman A was first proposed to in an earlier round than woman B. Once a woman is proposed to, she will be engaged (not necessarily with the same man) for the rest of the algorithm, so woman A's first engagement comes in an earlier round than woman B's.

If we just count the men's profiles we get 
\[\frac{n!^{n}}{2^n},\]
\[1,\ 27,\ 20736,\ 777600000,\ 2176782336000000,\ \ldots.\]

By symmetry, sequence \seqnum{A343692} counts the number of women's preferences where the first man is more desired than the second man. 

We also consider a symmetric case where both conditions apply: the first man is more desired than the second one, and the first woman is more desired than the second one. This sequence is the square of sequence \seqnum{A343692}, that is, the number of profiles is
\[\frac{n!^{2n}}{4^n}.\] 
It is now sequence \seqnum{A343693}: the number of preference profiles for $n$ men and $n$ women, $n>1$, such that all men prefer woman 1 over woman 2 and all women prefer man 1 over man 2.
\[1,\ 729,\ 429981696,\ 604661760000000000.\]
This sequence is sequence \seqnum{A338665} divided by $2^n$.

\subsection{Same taste}

Suppose the men have exactly the same taste; they rank all women in the same order. Then the number of ways to arrange men's preferences is $n!$: we just order women in the order of everyone's preference. This is sequence \seqnum{A000142}, the factorial numbers.

If we add women's preferences to this, we get a total of $(n!)^{n+1}$ profiles. This corresponds to the OEIS sequence \seqnum{A091868}. The sequence of preference profiles where all men have the same taste starts as:
\[1,\ 1,\ 8,\ 1296,\ 7962624,\ 2985984000000,\ \ldots.\]

Given such a profile, the Gale-Shapley men-proposing algorithm ends in exactly $n$ rounds.

If both men and women have the same taste, the number of such profiles is sequence \seqnum{A001044} of factorial squares, $n!^2$:
\[1,\ 4,\ 36,\ 576,\ 14400,\ 518400,\ 25401600,\ 1625702400,\ \ldots.\] 

Given such a profile, both the Gale-Shapley men-proposing and women-proposing algorithms end in exactly $n$ rounds. Indeed, in the first round of men-proposing algorithm, all men propose to the same woman: the homecoming queen. She gets engaged to the men of her first choice. In the second round, the $n-1$ unengaged men propose to the same women. This continues for $n$ rounds. Interestingly, we get couples with mutual rankings (1,1), (2,2), and so on. The egalitarian cost of the resulting matching is 
\[n(n+1).\]

\subsection{Tastes differ}

On the opposite side of the spectrum, there is another natural scenario where men prefer different women as their first choices. Such profiles have a chance to make more people happier.

For $n$ men and $n$ women, there are $n!$ ways for the men to arrange their first choices, and $(n-1)!^n$ ways for them to rank the rest of the women. The women's preferences do not matter, so there are $n!^n$ ways to choose them. This gives us a total of 
\[n!^{n+1}(n-1)!^n\]
preference profiles where all men prefer different women. By the way, these are exactly the profiles when the men-proposing Gale-Shapley algorithm terminates in one round. Thus, all men will be married to their first choices. Also, this sequence describes the number of profiles where men prefer different women as their $i$-th choice.

The sequence of preference profiles where all men prefer different women as their first choices, which is now \seqnum{A343475}, starts as
\[1, 8, 10368, 10319560704, 23776267862016000000,\ \ldots.\]

As before, we can create another sequence that just counts the number of possible preference profiles for men whose tastes differ. We just need to ignore women's preferences and divide this sequence by $n!^n$ to get the formula
\[n!(n-1)!^n.\]

We get another sequence which is now \seqnum{A343694}. It counts the number of men’s preference profiles where all men prefer different women as their first choices and starts as 
\[1, 2, 48, 31104, 955514880, 2149908480000000.\]

We can also consider a sequence where both men prefer different women and women prefer different men as their first choices. The new sequence counting such profiles, \seqnum{A343695}, is a square of the previous sequence:
\[n!^2(n-1)!^{2n},\] and starts as 
\[1, 4, 2304, 967458816, 913008685901414400.\]
For these profiles, both men-proposing and women-proposing algorithms end in one round.

We can restrict the profiles some more. Suppose not only men and women have different preferences for their first choices, but also these choices match each other. This is equivalent to saying that we have $n$ pairs of soulmates. This sequence is described in Section~\ref{sec:nsoul}.

\subsection{Tastes really differ}

Suppose now for any ranking $k$, the women that are ranked $k$ are all different. Previously, we only considered this for $k=1$. Now we do this for every choice. In other words, when arranged in a matrix, the men's preferences form a Latin square. Thus, the number of such profiles for men is described by sequence \seqnum{A002860} (the number of Latin squares) in the OEIS:
\[1,\ 2,\ 12,\ 576,\ 161280,\ 812851200,\ 61479419904000,\ \ldots.\]

Now we want to calculate the sequence in which women's profiles are also counted. There are still \seqnum{A002860}$(n)$ possibilities for the men's profiles and also $n!^n$ possibilities for the women's profiles. This gives us the sequence $n!^n$A002860$(n)$ of unique profiles for $n$ men and $n$ women where men's profiles form a Latin square, which is now \seqnum{A343696}:
\[1,\ 8,\ 2592,\ 191102976,\ 4013162496000000,\ 113241608573209804800000000, \ldots.\] 

As before, the men-proposing Gale-Shapely algorithm for such profiles terminates in one round.

Now we consider the sequence when both the men's and women's profiles form Latin squares. We call such profiles \textit{mutually-Latin} profiles. The formula for the number of profiles is \seqnum{A002860}$(n)^2$ because both the men and women have \seqnum{A002860}$(n)$ choices for profiles. The sequence counting preference profiles when both the men's and women's profiles form Latin squares is now \seqnum{A343697}:
\[1, 4, 144, 331776, 26011238400, 660727073341440000, 3779719071732351369216000000, \ldots.\]  
As before, both the men-proposing and the women-proposing Gale-Shapely algorithms for such profiles terminate in one round.

\subsection{Latin profiles}

The \textit{Latin} marriage problem is a subset of the stable marriage problem where the egalitarian cost of any man-woman pair is $n+1$ \cite{T2002}. The corresponding profiles are called \textit{Latin} profiles. The Latin marriage problem is interesting because its profiles tend to produce a lot of stable matchings. Sequence \seqnum{A069124} in the OEIS \cite{OEIS} describes the number of possible stable matchings in a Latin profile. It starts as
\[1,\ 2,\ 3,\ 10,\ 12,\ 32,\ 42,\ 268,\ 288,\ 656,\ 924,\ 4360,\ 3816,\ 11336,\ \ldots.\]

This sequence counts the number of stable matchings for Latin profiles, thus providing a lower bound in the maximum number of stable matchings for any profile for a given $n$.

In a Latin profile, men's preferences have to form a Latin square \cite{T2002}. The number of Latin profiles is the number of Latin squares of the corresponding order \seqnum{A002860}$(n)$ since a Latin square represents men's preferences, and after that, women's preferences are uniquely determined. 

\subsection{Disjoint profiles}

In this section, we study a special type of a profile called a \textit{disjoint} profile. In such profiles, each pair of mutual rankings $(i,j)$ occurs exactly once. In particular, each disjoint profile has exactly one pair of soulmates and one hell-pair. Such profiles correspond to entries in a disjoint-groups Sudoku \cite{STEPSudoku}, where in a particular place in a box, all digits are distinct across all the boxes. For example, the top-left corners of every box have distinct digits. This creates 9 additional groups to add to columns, rows, and boxes that have to contain distinct digits.

As we show in our other paper \cite{STEPSudoku}, there are 12 disjoint profiles for $n=2$. 

\begin{proposition}
For any $n$, the number of disjoint profiles does not exceed $n^2!$ and is divisible by $n!^2$.
\end{proposition}

\begin{proof}
There are $n^2$ possible rankings and $n^2$ pairs of people. Thus, each profile is a permutation of possible rankings to distribute them among pairs of possible people. The number of permutations is $n^2!$. Not all permutations can be used as we forbid ties in individual rankings. 

In such profiles, there is exactly one pair of soulmates. If we call them man 1 and woman 1, we can order everyone else according to their preferences. Thus each disjoint profile where man 1 and woman 1 rank their possible partners in order corresponds to $n!^2$ different profiles, which are achieved by relabeling men and women.
\end{proof}

The argument in the proof above allowed us to slightly speed up our program by assuming that the first man and the first woman rank other people in order and multiplying the resulting count of disjoint profiles by $n!^2$. We were able to calculate the number of disjoint profiles up to $n=4$. This is now sequence \seqnum{A345679}, which begins
\[1,\ 12,\ 8784,\ 1031049216,\ \ldots.\]

\subsection{Joint profiles}

In our paper \cite{STEPSudoku} we introduced a new type of Sudoku called \textit{joint-groups} or \textit{JG Sudoku}. In such a Sudoku every digit corresponds to a \textit{joint} preference profile.
In this new profile type, each possible mutual ranking $(i,j)$ appears exactly $n$ times, and once per $i$. We showed that a joint profile is a mutually-Latin profile; that is, both men's and women preferences form Latin squares. In addition, for a mutual ranking $(i,j)$, the value $j$ is a function of $i$: $j = f(i)$. We call the function $f$ the \textit{key} function.

The number of joint profiles is the number of Latin squares of order $n$ times the number of key functions, or $\text{A}002860(n)n!$. This is because a joint profile is uniquely defined by a Latin square of women's preferences and the key function, which is a permutation of $n$ elements.

The number of joint profiles as a function of $n$ is now sequence \seqnum{A344693}:
\[1,\ 4,\ 72,\ 13824,\ 19353600,\ 585252864000,\ 309856276316160000,\ \ldots.\]


\section{Soulmates}\label{sec:soulmates}

Recall that soulmates are a pair of people who rank each other first and who must therefore always be married.

To begin, let's look at the number of profiles with $k$ pairs of soulmates, which we denote by $F(k, n)$. For $n=2$, we can manually check to see that 2 profiles have no soulmates, 12 profiles have 1 pair of soulmates, and 2 profiles have 2 pairs of soulmates.

We introduce another function that is useful for our calculations. Given our set of $n$ men and $n$ women, consider a group of $k$ men and $k$ women. We ignore their preferences and calculate the number of profiles for the people not in our group, given that the people not in our group do not have soulmates. We denote this number as $S(k,n)$. We have
\[S(0,n) = F(0,n) \quad \textrm{ and } \quad S(n,n) = 1.\]

\begin{lemma}
The formula for $S(k,n)$ is 
\[S(k,n) = \sum_{i=0}^{n-k} (-1)^{i}\cdot \binom{n-k}{i}^2\cdot (n-1)!^{2i}\cdot i!\cdot n!^{2n-2k-2i}.\]
\end{lemma}

\begin{proof}
First, the total number of possible profiles for people not in the group is $n!^{2n-2k}$, since there are $n!$ ranking possibilities for each person, and a total of $2n-2k$ people. Next, we calculate the number of profiles that have at least one pair of soulmates and subtract it from the total number of profiles. 

There are ${n-k \choose 1}^2$ ways to pick a pair of soulmates out of $n-k$ men and $n-k$ women, and their profiles can be completed in $(n-1)!^2$ ways. Everyone else's profiles can be completed in $n!^{2n-2k-2}$ ways, giving a temporary total of 
\[{n-k \choose 1}^2 \cdot (n-1)!^2 \cdot n!^{2n-2k-2}.\]

However, this overcounts the number we want since each profile with more than 1 pair of soulmates is counted multiple times. For example, the profile with 2 pairs of soulmates is counted twice, so we must subtract the profiles where there are at least two pairs of soulmates, which would be ${n-k \choose 2} ^ 2 \cdot 2! \cdot (n-1)!^4 \cdot n!^{2n-2k-4}$. We continue in this fashion using the principle of inclusion-exclusion, and we get the complete formula:
\[S(k,n) = \sum_{i=0}^{n-k} (-1)^{i}\cdot \binom{n-k}{i}^2\cdot (n-1)!^{2i}\cdot i!\cdot n!^{2n-2k-2i}.\]
\end{proof}

As an example, for $k = n$, we get $S(n,n)=1$. If $k = n-1$ instead, we get
\[S(n-1,n) = n!^2 - (n-1)!^2 = (n-1)!^2(n^2-1).\]

Now we are ready to calculate our main function $F(k,n)$.

\begin{theorem}
The formula for $F(k,n)$ is 
\[\binom{n}{k}^2\cdot k! \cdot (n-1)!^{2k} \cdot \left(\sum_{i=0}^{n-k} (-1)^{i}\cdot \binom{n-k}{i}^2\cdot (n-1)!^{2i}\cdot i!\cdot n!^{2n-2k-2i}\right).\]
\end{theorem}

\begin{proof}
There are $\binom{n}{k}^2$ ways to choose $k$ pairs of people to be soulmates, and they can be matched to each other in $k!$ ways. These soulmates' profiles can be completed in $(n-1)!^{2k}$ ways. Now, all that remains is to complete the remaining $n-k$ people's profiles in such a way that there are no pairs of soulmates. We calculated this above as $S(k,n)$.

Thus, the complete formula for the number of profiles with exactly $k$ pairs of soulmates is 
\[F(k,n) = \binom{n}{k}^2\cdot k! \cdot (n-1)!^{2k} \cdot S(k,n),\]
or, in the expanded form,
\[\binom{n}{k}^2\cdot k! \cdot (n-1)!^{2k} \cdot \left(\sum_{i=0}^{n-k} (-1)^{i}\cdot \binom{n-k}{i}^2\cdot (n-1)!^{2i}\cdot i!\cdot n!^{2n-2k-2i}\right).\]
\end{proof}

We summarize the function $F(k,n)$ for small values of $k$ and $n$ in Table~\ref{table:ksoul}. Rows correspond to a fixed $n$, and columns correspond to a fixed $k$.

\begin{table}[ht!]
\begin{center}
\begin{tabular}{|c||c|c|c|c|c|}
\hline
& 0& 1& 2& 3& 4\\
\hline
1& 0& 1& & & \\
\hline
2& 2& 12& 2& & \\
\hline
3& 9984& 27072& 9216& 384& \\
\hline
4& 28419102720& 55736377344& 23460876288& 2418647040& 40310784\\
\hline
\end{tabular}
\end{center}
\caption{$F(k,n)$.}
\label{table:ksoul}
\end{table}

There are several sequences here that deserve our separate attention.

\subsection{Special cases}

\subsubsection{Love is in the air: $n$ pairs of soulmates}\label{sec:nsoul}

Suppose we have exactly $n$ pairs of soulmates. Our formula gives
\[F(n,n) = (n-1)!^{2n} n!.\]

The number of preference profiles for $n$ men and $n$ women with $n$ pairs of soulmates is now sequence \seqnum{A343698}. It starts as
\[1,\ 2,\ 384,\ 40310784,\ 7608405715845120, \ldots.\]

For such profiles, each person has exactly one valid partner: their soulmate. Consequently, there is only one stable matching: the one where the soulmates are married to each other. For these profiles, the Gale-Shapley algorithm ends in one round, whether it is men-proposing or women-proposing. Also, the resulting stable matching has the best possible egalitarian cost, which is $2n$.

\subsubsection{$n-1$ pairs of soulmates}

For $n-1$ pairs of soulmates the formula becomes
\[ (n-1)!^{2n+1} \cdot n^2 \cdot (n^2 - 1).\]
The sequence of preference profiles with $n-1$ pairs of soulmates is now sequence \seqnum{A343699}:
\[0,\ 12,\ 9216,\ 2418647040,\ 913008685901414400,\ \ldots.\] 

Such profiles have exactly one stable matching. Indeed, soulmates have to be matched with each other; after that, there are two people left, and they cannot avoid being together as a couple.
The men-proposing Gale-Shapley algorithm, when used on these preference profiles, ends in $j$ rounds, where $j$ is the ranking by the man in the non-soulmate pair of his partner. Why? Because each of the $n - 1$ men in the $n - 1$ soulmate pairs propose to their first choices, and then they are engaged for all the rounds. So, the only person left proposing is the $n$-th man in the non-soulmate pair. He gets to his true partner on the $j$-th round, where $j$ is where he ranks her.

\subsubsection{No soulmates}

To calculate the number of profiles without any pairs of soulmates we plug-in $k=0$, into the formula for $S(k,n)$. We get
\[\binom{n}{0}^2\cdot 0! \cdot (n-1)!^{0} \cdot \left(\sum_{i=0}^{n-0} (-1)^{i}\cdot \binom{n-0}{i}^2\cdot (n-1)!^{2i}\cdot i!\cdot n!^{2n-20-2i}\right).\]
Simplifying, we have
\[\sum_{i=0}^{n} (-1)^{i}\cdot \binom{n}{i}^2\cdot (n-1)!^{2i}\cdot i!\cdot n!^{2n-2i}.\]

The number of profiles without soulmates is now sequence \seqnum{A343700}, which starts as
\[0,\ 2,\ 9984,\ 28419102720,\ 175302739963548794880,\ \ldots.\]


Such profiles are useful for finding the bound for $M(n)$, which is the maximum possible number of stable matchings. As the sequence $M(n)$ is growing \cite{T2002}, the value $M(n)$ can only be achieved if the profile does not have soulmates.

\subsection{Tastes really differ}

Let us go back to profiles where men's preferences form a Latin square and see what happens if we require $k$ pairs of soulmates. We denote the number of Latin squares of size $n$ as $L_n$.

\begin{theorem}
The number of preference profile in a stable matching problem with $n$ men and $n$ women, such that there are exactly $k$ pairs of soulmates is
\[L_n \cdot \binom{n}{k}\cdot (n-1)^{n-k} \cdot (n-1)!^n.\]
\end{theorem}

\begin{proof}
Since the preference profiles for men form a Latin square, every woman will have exactly one man that ranks her first. Therefore, if we want $k$ pairs of soulmates, there are $\binom{n}{k}$ ways to pick the $k$ women for the soulmates pairs. After that, we have $(n-1)!$ ways to complete the profiles for every woman that has a soulmate, and $(n-1)(n-1)!$ ways to complete the profiles for every woman that does not have a soulmate. Since we have $L_n$ ways to choose the profiles for men, this gives us a total of 
\[L_n \cdot \binom{n}{k}\cdot (n-1)!^k \cdot (n-1)^{n-k}(n-1!)^{n-k} = L_n \cdot \binom{n}{k}\cdot (n-1)^{n-k} \cdot (n-1)!^n\] preference profiles that have $k$ pairs of soulmates.
\end{proof}

We summarize the results for the small values of $n$ and $k$ in Table~\ref{table:mtrd}. The columns correspond to different $n$ and rows correspond to different $k$.

\begin{table}[ht!]
\begin{center}
\begin{tabular}{|c|c|c|c|c|c|}
\hline
 &1&2&3&4&5\\
\hline
0&0&2&768&60466176&1315033086689280\\
\hline
1&1&4&1152&80621568&1643791358361600\\
\hline
2&---&2&576&40310784&821895679180800\\
\hline
3&---&---&96&8957952&205473919795200\\
\hline
4&---&---&---&746496&25684239974400\\
\hline
5&---&---&---&---&1284211998720\\
\hline
Total&1&8&2592&191102976&4013162496000000\\
\hline
\end{tabular}
\end{center}
\caption{Distribution of preference profiles by the number of soulmate pairs when men's tastes really differ.}
\label{table:mtrd}
\end{table}

As expected, the total for a given $n$ is sequence \seqnum{A343696}, which counts the number of profiles when men's tastes really differ. We also have two more sequences of note. The first sequence is the number of preference profiles when men's tastes really differ and there are $n$ soulmates. It corresponds to the case $k=n$, and the formula for it is 
\[L_n \cdot (n-1)!^n.\]
The number of preference profiles when men's tastes really differ and there are $n$ pairs of soulmates is now sequence \seqnum{A344662}:
\[1,\ 2,\ 96,\ 746496,\ 1284211998720,\ 2427160677580800000000,\ \ldots.\]

Another sequence of note is the sequence counting the number of profiles where the men's tastes really differ and there are 0 soulmates. In this case, the formula is 
\[L_n \cdot (n-1)^n \cdot (n-1)!^n.\]
The sequence which counts the number of profiles where the men's taste really differ, with 0 soulmate pairs, is now sequence \seqnum{A344663}:
\[0, 2, 768, 60466176, 1315033086689280, 37924385587200000000000000,\ \ldots.\]

Now we look at the case when both men's and women's preference profiles form Latin squares, and where there are $k$ pairs of soulmates.

\begin{theorem}
The number of preference profiles for $n$ men and $n$ women where both men's and women's preference profiles form Latin squares and there are $k$ pairs of soulmates is
\[\frac{L_n^2}{n!} \cdot \binom{n}{k} \cdot \left( \sum_{i=0}^{n-k} (-1)^i \frac{(n-k)!}{i!} \right).\]
\end{theorem}

\begin{proof}
We have $L_n$ ways to choose men's preference profiles. Now we calculate the number of ways for women to pick their first choices.

We have $\binom{n}{k}$ ways to pick the women in pairs of soulmates. After that, we need to make sure that there are no other soulmate pairs. Among the $n-k$ women without soulmates, the people that they rank first will be some permutation of the $n-k$ men that are not soulmates. This is because in a Latin preference profile, all of the people that are ranked first must be different, and all of the men that are in a soulmate pair have already been ranked first by someone, so they cannot be ranked first again. We need these $n-k$ men and women to not have any pairs of soulmates, so we need our permutation to be a derangement. We denote the number of derangements of size $i$ as $D_i$.

This gives us that there are $\binom{n}{k} \cdot D_{n-k}$ possible first choices for women.
They can be completed in $\frac{L_n}{n!}$ ways. Therefore, we have $\frac{L_n}{n!} \cdot \binom{n}{k} \cdot D_{n-k}$ preference profiles for women, which gives us a total of 
\[\frac{L_n^2}{n!} \cdot \binom{n}{k} \cdot  D_{n-k} = \frac{L_n^2}{n!} \cdot \binom{n}{k} \cdot \left( \sum_{i=0}^{n-k} (-1)^i \frac{(n-k)!}{i!} \right)\]
preference profiles with $k$ soulmates. 
\end{proof}

We summarize the results for small values of $n$ in Table~\ref{table:mwtrd}. The columns correspond to different $n$ and rows to different $k$.

\begin{table}[ht!]
\begin{center}
\begin{tabular}{|c|c|c|c|c|c|}
\hline
 &1&2&3&4&5\\
\hline
0&0&2&48&124416&9537454080\\
\hline
1&1&0&72&110592&9754214400\\
\hline
2&*&2&0&82944&4335206400\\
\hline
3&*&*&24&0&2167603200\\
\hline
4&*&*&*&13824&0\\
\hline
5&*&*&*&*&216760320\\
\hline
Total&1&4&144&331776&26011238400\\
\hline
\end{tabular}
\end{center}
\caption{Distribution by the number of soulmates when men's and women's tastes really differ.}
\label{table:mwtrd}
\end{table}

One might notice the zeros on the diagonal. The reason is that when both men's and women's preferences form Latin squares, every single man and woman must be ranked first exactly once. Therefore, we cannot have exactly $n-1$ pairs of soulmates. This is also reflected in the formula since if we want to have $n-1$ pairs of soulmates, we would be multiplying by the number of derangements of size 1, which is 0.

As expected, the total in each column is sequence \seqnum{A343697}. But we also have two more sequences of note. The first sequence is the number of preference profiles when men's and women's tastes really differ and there are $n$ pairs of soulmates. It corresponds to the case $k=n$, and the formula simplifies to 
\[\frac{L_n^2}{n!}.\]
This is the number of Latin squares multiplied by the number of Latin squares up to relabeling the digits. The sequence counting the number of preference profiles where both men's and women's preferences form Latin squares, and there are $n$ pairs of soulmates is now sequence \seqnum{A344664}, and it starts as
\[1,\ 2,\ 24,\ 13824,\ 216760320,\ 917676490752000,\ 749944260264355430400000,\ \ldots.\]

Additionally, in the case with no pairs of soulmates, then the formula simplifies to:
\[\frac{L_n^2}{n!} \cdot \left( \sum_{i=0}^n (-1)^i \frac{n!}{i!} \right).\]
The sequence counting the number of preference profiles where both men and women's preferences form Latin squares and there no soulmates is now \seqnum{A344665}; and it starts as
\[0,\ 2,\ 48,\ 124416,\ 9537454080,\ 243184270049280000,\ \ldots.\]

\section{Distribution for the number of stable matchings}\label{sec:distSM}

Given a preference profile, we can calculate the total number of stable matchings.

\textbf{Example.} Here is an example for a preference profile for $n=2$ that allows two stable matchings: First man: [1, 2], second man: [2, 1], first woman: [2, 1], second woman: [1, 2]. There are two possible matchings, and both are stable.

We wrote a program to calculate the number of stable matchings for every profile for $n \leq 4$. The results are summarized in Table~\ref{table:distribution}, where the columns correspond to different $n$. Row $k$ shows how many different profiles have exactly $k$ stable matchings. For example, there are 144 profiles for $n=4$ that result in 10 distinct stable matchings.

\begin{table}[ht]
\centering
\begin{tabular}{|c|c|c|c|c|}
\hline 
       & 1 & 2 & 3 & 4 \\
\hline
1	& 1 	& 14 	& 34080 & 65867261184 \\
\hline
2	& 0 	& 2	& 11484	& 35927285472	\\
\hline
3	& 0 	& 0	& 1092	& 7303612896 \\
\hline
4	& 0 	& 0	& 0	& 861578352\\
\hline
5	& 0	& 0 	& 0	&111479616 \\
\hline
6	& 0 	& 0 	& 0	& 3478608 \\
\hline
7	& 0	& 0	& 0	&581472	\\
\hline
8	& 0	& 0	& 0	& 36432 \\
\hline
9	&0	& 0	& 0	& 0 \\
\hline
10	&0 	& 0	& 0	& 144 \\
\hline
Total 	& 1	& 16	& 46656	& 110075314176 \\
\hline
\end{tabular}
\caption{Distribution for the number of stable matchings.}
\label{table:distribution}
\end{table}

We submitted four sequences to the OEIS related to this table. 

The first sequence is the number of preference profiles for 3 men and 3 women that generate $n$ possible stable matchings, which is now sequence \seqnum{A344666}. It is a finite sequence of three terms:
\[34080,\ 11484,\ 1092.\]
The second sequence, which is now \seqnum{A344667}, is the number of preference profiles for 4 men and 4 women that generate $n$ possible stable matchings. It is a finite sequence with 10 terms:
\[65867261184,\ 35927285472,\ 7303612896,\ 861578352,\ 111479616,\ 3478608,\ 581472,\ 36432, 0,\ 144.\]

We also submitted two infinite sequences. The first sequence, now sequence \seqnum{A344668}, corresponds to the first row. It is the number of profiles for $n$ men and $n$ women that result in one stable matching:
\[1,\ 14,\ 34080,\ 65867261184,\ \ldots.\]
The second sequence, now sequence \seqnum{A344669}, is the number of profiles for $n$ men and $n$ women that result in the maximum number of stable matchings. It is the last non-zero term in each column:
\[1,\ 2,\ 1092,\ 144,\ \ldots.\]

\subsection{Soulmates}

Suppose we have a pair of soulmates. Since they have to be married to each other, we can remove them from consideration and look at the other couples. It follows that the maximum number of stable matchings when there is exactly one pair of soulmates is $M(n-1)$. Similarly, if there are $k$ pairs of soulmates, the maximum possible number of stable matchings is $M(n-k)$.

For example, we know that for $n=2$, the maximum possible number of stable matchings is 2. This means for $n=3$, the maximum number of stable matchings with at least one couple of soulmates is also 2.

Table~\ref{table:SMVersusGoodness} shows the number of profiles for a given $n$ where there is a soulmate pair, and a given number of stable matchings. The leftmost column represents the number of stable matchings, while the top row is $n$. 

\begin{table}[ht]
\centering
\begin{tabular}{|c|c|c|c|}
	\hline
	& 1 & 2 & 3 \\
	\hline
	1 & 1 & 14 & 30840\\
	\hline
	2 & 0 & 0 & 5832\\
	\hline
	Total & 1 & 14 & 36672 \\
\hline
\end{tabular}
\caption{Distribution for the number of stable matchings with at least one pair of soulmates.}
\label{table:SMVersusGoodness}
\end{table}

\section{Special Cases}\label{sec:special}

\subsection{A hell-pair and a hell-couple}

Suppose we have two people that rank each other as $n$. We call such a pair a \textit{hell-pair}. If they end up getting married, we call them a \textit{hell-couple}. Such a pair has an egalitarian cost of $2n$. Replacing ranking $i$ with $n+1-i$ is a bijection that swaps soulmates with hell-pairs. Because of that, the sequence counting the number of profiles with $k$ hell-pairs is the same as the corresponding sequence with $k$ pairs of soulmates.

Is it possible to have a hell-couple in a stable matching? Yes, it is. For example, two people that are ranked last by everyone else have to be together no matter what mutual ranking they have. This means any mutual ranking of two people in a married couple is possible.

\begin{proposition}
In a stable matching, it is impossible to have more than 1 hell-couple.
\end{proposition}

\begin{proof}
Suppose we have two hell-couples. Two people of different genders from two different couples form a rogue couple.
\end{proof}

Suppose there is a hell-couple in a stable matching. A person in this hell situation is ready to run away with anyone who agrees. This means that in a stable matching, every other person values the member of the hell-couple of the opposite gender less than their partner. In particular, no one can rank first a member of a hell-couple.

\begin{proposition}\label{prop:hellcouple}
If a preference profile has a stable matching with a hell-couple, all possible stable matchings have that hell couple.
\end{proposition}

\begin{proof}
We have a preference profile with a stable matching in which two people, let us say man $A$ and woman $Z$, form a hell-couple. We know that the men-proposing Gale-Shapley algorithm is woman-pessimal, meaning that women are always matched with the man they prefer least out of all their valid partners. This means that if we were to apply the men-proposing Gale-Shapley algorithm to this preference profile, we must have $Z$ matched with $A$, since $Z$ prefers $A$ least and we are given that they are valid partners. Additionally, the men-proposing Gale-Shapely algorithm is man-optimal, meaning that men are always matched with the person they prefer most out of their valid partners. This means $Z$ is the only valid partner for $A$. Similarly, $A$ is the only valid partner for $Z$. Thus, they form a couple in any stable matching.
\end{proof}

It follows that if a profile's stable matching has a hell-couple, the possible number of matchings for such a profile does not exceed $M(n-1)$.

Consider the sequence describing the number of profiles such that there exists a stable matching with a hell-couple. We calculated this sequence, which is now sequence \seqnum{A344670}, by a program, and it starts as:
\[1,\ 4,\ 4536,\ 5113774080,\ \ldots.\]

Consider pairs containing a preference profile and one of its stable matchings. We calculated the sequence of the number of these pairs that generate a hell-couple. That is, each profile that generates a hell-couple in a matching was counted with multiplicity equaling the number of matchings that generate this hell-couple. As per Proposition~\ref{prop:hellcouple} the multiplicity is the number of possible stable matchings for this profile.
The sequence, which is now sequence \seqnum{A344671}, starts as:
\[1,\ 4,\ 4608,\ 5317484544,\ \ldots.\]

\subsection{Outcasts}

It might be that the opposite notion to soulmates is not a hell-pair, but rather the outcasts. We call two people the \textit{outcasts} if they are ranked $n$ by everyone else: people of the opposite gender that are not in this pair. In a stable matching, two outcasts have to be paired with each other. This is independent of their mutual rankings, so they might be soulmates or a hell-couple, or anything in between.

We can have at most one pair of outcasts (unless $n = 2$), since everyone except the two outcasts ranks them last. We want to calculate the number of profiles with at least one pair of outcasts.

\begin{proposition}
The number of preference profiles with $n$ men and $n$ women such that there exists at least one pair of outcasts is:
\[n^4 \cdot (n-1)!^{2n} \textrm{ for } n \neq 2, \textrm{ and } 14  \textrm{ for } n = 2.\]
\end{proposition}

\begin{proof}
The statement is true for $n=1$. Suppose $n>2$. Then we have $\binom{n}{1}^2 = n^2$ ways to pick the two outcasts, then $n!^2$ ways to complete the outcasts' preference profiles, and finally $(n-1)!^{2n-2}$ ways to complete everyone else's profile; therefore we have 
\[n^2 \cdot n!^2 \cdot (n-1)!^{2n-2} = n^4 \cdot (n-1)!^{2n}\]
ways to have a pair of outcasts for $n>2$. 

However, when $n = 2$, it is possible for us to have two outcast pairs, and we are double-counting these cases. We can pick the two pairs in $2! = 2$ ways, and from there we only have 1 way to complete the preference profile, so there are 2 preference profiles that have two pairs of outcasts. Subtracting these from the previous number to avoid double-counting, we get that there are $14$ preference profiles that have at least one pair of outcasts for $n=2$. There is another description of profiles that do not contain outcasts: These are the Latin profiles where each pair of people have the egalitarian cost of 3. There are only 2 such profiles.

The formula for the number of outcasts is:
\[n^4 \cdot (n-1)!^{2n} \textrm{ for } n \neq 2, \textrm{ and } 14  \textrm{ for } n = 2.\]
\end{proof}

The first few terms of this sequence counting the number of preference profiles that have outcasts, which is now sequence \seqnum{A344689}, are:
\[1,\ 14,\ 5184,\ 429981696,\ 39627113103360000,\ \ldots.\]

\subsection{Outcast-hell pair}

In order for us to have an outcast-hell pair, we need to pick two people, and have everyone rank them last. Note that this means there can only be one outcast-hell-pair. Moreover, in such scenario we can have at most one hell-pair, and the people in it have to be outcasts. We can pick two people to be outcasts in $\binom{n}{1}^2 = n^2$ ways, and since everyone ranks them last, we can complete everyone's preference profiles in $(n-1)!^{2n}$ ways, so for any $n$, we have $n^2 \cdot (n-1)!^{2n}$ preference profiles with an outcast-hell pair. We saw this sequence before. By symmetry, this is the same as the number of profiles where there is pair of people such that everyone ranks them first. This is sequence \seqnum{A343474}:
\[1,\ 4,\ 576,\ 26873856,\ 1585084524134400,\ \ldots.\]

Except for the second term, the number of outcast profiles equals $n^2$ times the number of outcast-hell pair profiles. This makes sense, as the only difference is that the members of the outcast pair can rank each other any way they want.

\section{Symmetries}\label{sec:symmetries}

We can relabel men, relabel women and we can swap men with women. 

We can relabel men in $n!$ ways. This is equivalent to fixing the first women's preferences to be 1, 2, 3, and so on. In particular, the total number of profiles has to be divisible by $n!$. 

In addition, we can relabel the women, excluding the first, in the order of the preferences of the first man. This effectively means that we can shuffle the men and all the women but the first. Thus we have groups of $n!(n-1)!$ profiles which result in the same matchings but use different orderings of men and women.

\subsection{Distinct men's profiles up to relabeling}

We calculate the number of different men's profiles up to relabeling men for small values of $n$.

Suppose $n = 1$. Then, there is only one woman and only one man, so there can only be one possible set of lists.

Suppose $n = 2$. If both men prefer the same woman, this woman can be either woman 1 or 2. If they prefer different women, we can shuffle them, so that the first man prefers woman 1. Thus, the total is 3.


Here is the general formula. There are $n!$ possible profiles and $n$ men, so we are choosing $n$ elements from a list of $n!$ elements, with repeats allowed, and the order in which the elements are chosen does not matter. Therefore, the total number of ways to do so is
\[\binom{n!+n-1}{n} .\]

The sequence of the number of different men's profiles up to relabeling men is now sequence \seqnum{A344690} and it starts as:
\[1,\ 3,\ 56,\ 17550,\ 225150024,\ 197554684517400,\ 16458566311785642529680, \ldots.\]

We can think of this sequence as the number of $n$-multisets chosen from the set of all permutations of $n$ elements.

\section{The egalitarian cost distribution for small values of $n$}\label{sec:ec}

The egalitarian cost depends on a matching. Given a preference profile, it is possible to have different stable matchings with different egalitarian costs. Interestingly, if $n=2$ or 3 and there are several stable matchings, the egalitarian cost is the same for all these matchings. This does not hold for larger $n$.

Suppose $n=2$. If there is one stable matching, there are 2 profiles with an egalitarian cost of 4, 8 profiles with an egalitarian cost of 5, and 4 profiles with an egalitarian cost of 6. If there are two stable matchings, there are two profiles with an egalitarian cost of 6. Thus, we have two sequences. The number of profiles such that the matching has an egalitarian cost $g$: 2, 8, 6. The number of possible matchings for different profiles such that the egalitarian cost is $g$: 2, 8, 8.

Suppose $n = 3$. We know that 34080 profiles have one stable matching, 11484 profiles have two matchings, and 1092 profiles have 3 matchings.

Possible egalitarian costs range from 6 to 12, inclusive. If the profile has a matching with an egalitarian cost of 6, then this matching is unique, and all the couples are soulmates. We already calculated the number of such profiles as 384. If the egalitarian cost is 7 for one of the matchings, then the matching is unique too. Indeed, we have two soulmate couples that have to be together.

We cannot have a total egalitarian cost of 18. Suppose there is a matching of egalitarian cost 18. Then all couples are hell-couples. In this case, any pair of people that are not in a couple form a rogue pair.

When a profile has multiple matchings, it is possible that the matchings have different total egalitarian costs. For example, consider a profile, where the following matrices describe women's and men's preferences, respectively:
\[
\begin{pmatrix}
 3  & 1 & 2 \\ 
 1  & 2 & 3 \\ 
 1  & 2 & 3 
\end{pmatrix}
 \quad 
\begin{pmatrix}
 1  & 2 & 3 \\ 
 2  & 3 & 1 \\ 
 3  & 2 & 1 
\end{pmatrix}.
\]

The men-proposing Gale-Shapley algorithm ends with a total egalitarian cost of 12.  The couples are the following: man 1 and woman 1, man 2 and woman 3, and man 3 and woman 2. The women-proposing algorithm ends with a total egalitarian cost of 11. The couples are the following: woman 1 and man 3, woman 2 and man 1, and woman 3 and man 2.

We calculated the number of preference profiles in the stable marriage problem with 3 men and 3 women such that there exists a stable matching with egalitarian cost $n$. The egalitarian cost for stable matchings ranges from 6 to 12 inclusive. The numbers below were calculated by a program:
\[384,\ 2304,\ 7416,\ 13860,\ 15912,\ 10836,\ 3564.\]

We did a similar calculation for $n=4$. The egalitarian cost in this case ranges from 8 to 20: 40310784, 322486272, 1394454528, 4263542784, 9856161792, 17805053952, 25557163776, 29223099648, 26437927680, 18541903680, 9633334320, 3379380192, 626260608.

We combined all these numbers into one sequence \seqnum{A344691}: an irregular table, where $T(n,k)$ is the number of preference profiles in the stable marriage problem with $n$ men and $n$ women such that there exists a stable matching with an egalitarian cost of $k$.

One might notice that the sum $T(3,k)$ over different $k$ is 54276, which is bigger than the total number of preference profiles, which is 46656. That means that we have the profiles that give matchings with different egalitarian costs. We know that 34080 profiles have one stable matching, 11484 profiles have two matchings, and 1092 profiles have 3 matchings. The total number of matchings with different profiles is:
\[1 \cdot 34080 + 2 \cdot 11484 + 3 \cdot 1092 = 60324.\]

It makes it natural to calculate another sequence that counts different matchings with a given egalitarian cost rather than different profiles. In other words, for a given egalitarian cost and each profile that can generate a matching with this egalitarian cost, we count this profile with the weight equaling the number of different matchings with the given egalitarian cost. 

The irregular table $T'(n,k)$ is combined in sequence \seqnum{A344692}: the number of different stable matchings per different profiles for $n$ men and $n$ women with an egalitarian cost of $k$. For $n=3$ and egalitarian costs ranging from 6 to 12 we get the following numbers:
\[384,\ 2304,\ 7488,\ 14592,\ 18072,\ 13104,\ 4380.\]
The total as expected equals 60324. For $n=4$ the egalitarian cost ranges from 12 to 20 and the numbers are: 40310784, 322486272, 1397440512, 4299816960, 10080681984, 18632540160, 27586068480, 32664453120, 30544625664, 21941452800, 11480334336, 3963617280, 707788800.
It is not surprising that 
\[T(n,2n) = T'(n,2n) \quad \textrm{ and } \quad T(n,2n+1) = T'(n,2n+1) .\]
If the egalitarian cost is $2n$, correspondingly $2n+1$, there are $n$, correspondingly $n-1$ pairs of soulmates. In both cases, there is only one stable matching per profile.

\section{Acknowledgments}

This project was done as part of MIT PRIMES STEP, a program that allows students in grades 6 through 9 to try research in mathematics. Tanya Khovanova is the mentor of this project. We are grateful to PRIMES STEP for this opportunity.

Mentions 
\seqnum{A000142}
\seqnum{A001044}
\seqnum{A002860}
\seqnum{A069124}
\seqnum{A091868}
\seqnum{A185141}

New sequences 
\seqnum{A338665}
\seqnum{A340890}
\seqnum{A342573}
\seqnum{A343474}
\seqnum{A343475}
\seqnum{A343692}
\seqnum{A343693}
\seqnum{A343694}
\seqnum{A343695}
\seqnum{A343696}
\seqnum{A343697}
\seqnum{A343698}
\seqnum{A343699}
\seqnum{A343700}
\seqnum{A344662}
\seqnum{A344663}
\seqnum{A344664}
\seqnum{A344665}
\seqnum{A344666}
\seqnum{A344667}
\seqnum{A344668}
\seqnum{A344669}
\seqnum{A344670}
\seqnum{A344671}
\seqnum{A344689}
\seqnum{A344690}
\seqnum{A344691}
\seqnum{A344692}
\seqnum{A344693}
\seqnum{A345679}

\end{document}